\documentclass[a4paper,11pt,reqno]{amsart}

\usepackage[utf8x]{inputenc}
\usepackage{hyperref}
\usepackage{xcolor}
\hypersetup{
    colorlinks,
    linkcolor={red!50!black},
    citecolor={blue!50!black},
    urlcolor={blue!80!black}
}

\usepackage{microtype}
\usepackage{MnSymbol}
\usepackage{textcomp}
\usepackage{stmaryrd}

\newtheoremstyle{plain}
  {.5\baselineskip\@plus.2\baselineskip\@minus.2\baselineskip}
  {.5\baselineskip\@plus.2\baselineskip\@minus.2\baselineskip\@plus.5em}
  {\slshape}
  {}
  {\bfseries}
  {.}
  { }
  {}

\newtheoremstyle{definition}
  {.5\baselineskip\@plus.2\baselineskip\@minus.2\baselineskip}
  {0.8\baselineskip\@plus.2\baselineskip\@minus.2\baselineskip\@plus.5em}
  {}
  {}
  {\bfseries}
  {.}
  { }
  {}

\makeatletter
\newcommand{\eqnum}{\refstepcounter{equation}\textup{\tagform@{\theequation}}}
\makeatother

\newtheorem{thm}{Theorem}
\newtheorem{thmX}{Theorem}

\newtheorem*{thm*}{Theorem}

\newtheorem{prop}{Proposition}
\newtheorem*{defthm*}{Definition/Theorem}

\theoremstyle{definition}

\newtheorem{rem}{Remark}

\newtheorem*{exam*}{Example}

\newcommand\arXiv[1]{\href{http://arxiv.org/abs/#1}{arXiv:#1}}

\usepackage{xparse}

\usepackage{xspace}

\newcommand{\nc}{\newcommand}
\nc{\renc}{\renewcommand}
\nc{\ssec}{\subsection}
\nc{\sssec}{\subsubsection}
\nc{\on}{\operatorname}
\nc{\term}[1]{#1\xspace}

\usepackage{tikz}
\usetikzlibrary{matrix}
\usepackage{tikz-cd}

\tikzset{
  commutative diagrams/.cd,
  arrow style=tikz,
  diagrams={>=latex}}
\makeatletter
\tikzset{
  column sep/.code=\def\pgfmatrixcolumnsep{\pgf@matrix@xscale*(#1)},
  row sep/.code   =\def\pgfmatrixrowsep{\pgf@matrix@yscale*(#1)},
  matrix xscale/.code=%
    \pgfmathsetmacro\pgf@matrix@xscale{\pgf@matrix@xscale*(#1)},
  matrix yscale/.code=%
    \pgfmathsetmacro\pgf@matrix@yscale{\pgf@matrix@yscale*(#1)},
  matrix scale/.style={/tikz/matrix xscale={#1},/tikz/matrix yscale={#1}}}
\def\pgf@matrix@xscale{1}
\def\pgf@matrix@yscale{1}
\makeatother

\usepackage[cmtip, all]{xy}

\usepackage[inline]{enumitem}
\setlist[enumerate,1]{label={(\alph*)},itemsep=\parskip}

\newlist{thmlist}{enumerate}{2}
\setlist[thmlist,1]{
  label={\em(\roman*)}, ref={(\roman*)},
  itemsep=0.5em,
  align=right,widest=vi)}
\setlist[thmlist,2]{
  label={\em(\alph*)}, ref={(\alph*)},
  itemsep=0.75em,
  labelsep=0em,labelindent=0em,leftmargin=*,align=left,widest=vi),
  topsep=0.75em}

\newlist{thmlistbis}{enumerate}{1}
\setlist[thmlistbis,1]{
  label={\em(\roman*~\textit{bis})},
  ref={(\roman*}~\textit{bis}\upshape{)},
  itemsep=0.5em,
  leftmargin=0pt, align=right, widest=vi)}

\newlist{defnlist}{enumerate}{2}
\setlist[defnlist,1]{
  label={(\roman*)}, ref={(\roman*)},
  itemsep=0.5em,
  align=right, widest=vi)}

\setlist[defnlist,2]{
  label={(\alph*)}, ref={(\alph*)},
  itemsep=0.75em,
  labelsep=0em,labelindent=0em,leftmargin=*,align=left,widest=vi),
  topsep=0.75em}

\newlist{inlinelist}{enumerate*}{1}
\setlist[inlinelist,1]{label={(\alph*)}}

\newlist{inlinedefnlist}{enumerate*}{1}
\definecolor{green}{HTML}{38550C}
\setlist[inlinedefnlist,1]{label={\color{green}(\roman*)}}

\nc{\sA}{\ensuremath{\mathcal{A}}\xspace}
\nc{\sB}{\ensuremath{\mathcal{B}}\xspace}
\nc{\sC}{\ensuremath{\mathcal{C}}\xspace}
\nc{\sD}{\ensuremath{\mathcal{D}}\xspace}
\nc{\sE}{\ensuremath{\mathcal{E}}\xspace}
\nc{\sF}{\ensuremath{\mathcal{F}}\xspace}
\nc{\sG}{\ensuremath{\mathcal{G}}\xspace}
\nc{\sH}{\ensuremath{\mathcal{H}}\xspace}
\nc{\sI}{\ensuremath{\mathcal{I}}\xspace}
\nc{\sJ}{\ensuremath{\mathcal{J}}\xspace}
\nc{\sK}{\ensuremath{\mathcal{K}}\xspace}
\nc{\sL}{\ensuremath{\mathcal{L}}\xspace}
\nc{\sM}{\ensuremath{\mathcal{M}}\xspace}
\nc{\sN}{\ensuremath{\mathcal{N}}\xspace}
\nc{\sO}{\ensuremath{\mathcal{O}}\xspace}
\nc{\sP}{\ensuremath{\mathcal{P}}\xspace}
\nc{\sQ}{\ensuremath{\mathcal{Q}}\xspace}
\nc{\sR}{\ensuremath{\mathcal{R}}\xspace}
\nc{\sS}{\ensuremath{\mathcal{S}}\xspace}
\nc{\sT}{\ensuremath{\mathcal{T}}\xspace}
\nc{\sU}{\ensuremath{\mathcal{U}}\xspace}
\nc{\sV}{\ensuremath{\mathcal{V}}\xspace}
\nc{\sW}{\ensuremath{\mathcal{W}}\xspace}
\nc{\sX}{\ensuremath{\mathcal{X}}\xspace}
\nc{\sY}{\ensuremath{\mathcal{Y}}\xspace}
\nc{\sZ}{\ensuremath{\mathcal{Z}}\xspace}

\nc{\bA}{\ensuremath{\mathbf{A}}\xspace}
\nc{\bB}{\ensuremath{\mathbf{B}}\xspace}
\nc{\bC}{\ensuremath{\mathbf{C}}\xspace}
\nc{\bD}{\ensuremath{\mathbf{D}}\xspace}
\nc{\bE}{\ensuremath{\mathbf{E}}\xspace}
\nc{\bF}{\ensuremath{\mathbf{F}}\xspace}
\nc{\bG}{\ensuremath{\mathbf{G}}\xspace}
\nc{\bH}{\ensuremath{\mathbf{H}}\xspace}
\nc{\bI}{\ensuremath{\mathbf{I}}\xspace}
\nc{\bJ}{\ensuremath{\mathbf{J}}\xspace}
\nc{\bK}{\ensuremath{\mathbf{K}}\xspace}
\nc{\bL}{\ensuremath{\mathbf{L}}\xspace}
\nc{\bM}{\ensuremath{\mathbf{M}}\xspace}
\nc{\bN}{\ensuremath{\mathbf{N}}\xspace}
\nc{\bO}{\ensuremath{\mathbf{O}}\xspace}
\nc{\bP}{\ensuremath{\mathbf{P}}\xspace}
\nc{\bQ}{\ensuremath{\mathbf{Q}}\xspace}
\nc{\bR}{\ensuremath{\mathbf{R}}\xspace}
\nc{\bS}{\ensuremath{\mathbf{S}}\xspace}
\nc{\bT}{\ensuremath{\mathbf{T}}\xspace}
\nc{\bU}{\ensuremath{\mathbf{U}}\xspace}
\nc{\bV}{\ensuremath{\mathbf{V}}\xspace}
\nc{\bW}{\ensuremath{\mathbf{W}}\xspace}
\nc{\bX}{\ensuremath{\mathbf{X}}\xspace}
\nc{\bY}{\ensuremath{\mathbf{Y}}\xspace}
\nc{\bZ}{\ensuremath{\mathbf{Z}}\xspace}

\nc{\bbA}{\ensuremath{\mathbb{A}}\xspace}
\nc{\bbB}{\ensuremath{\mathbb{B}}\xspace}
\nc{\bbC}{\ensuremath{\mathbb{C}}\xspace}
\nc{\bbD}{\ensuremath{\mathbb{D}}\xspace}
\nc{\bbE}{\ensuremath{\mathbb{E}}\xspace}
\nc{\bbF}{\ensuremath{\mathbb{F}}\xspace}
\nc{\bbG}{\ensuremath{\mathbb{G}}\xspace}
\nc{\bbH}{\ensuremath{\mathbb{H}}\xspace}
\nc{\bbI}{\ensuremath{\mathbb{I}}\xspace}
\nc{\bbJ}{\ensuremath{\mathbb{J}}\xspace}
\nc{\bbK}{\ensuremath{\mathbb{K}}\xspace}
\nc{\bbL}{\ensuremath{\mathbb{L}}\xspace}
\nc{\bbM}{\ensuremath{\mathbb{M}}\xspace}
\nc{\bbN}{\ensuremath{\mathbb{N}}\xspace}
\nc{\bbO}{\ensuremath{\mathbb{O}}\xspace}
\nc{\bbP}{\ensuremath{\mathbb{P}}\xspace}
\nc{\bbQ}{\ensuremath{\mathbb{Q}}\xspace}
\nc{\bbR}{\ensuremath{\mathbb{R}}\xspace}
\nc{\bbS}{\ensuremath{\mathbb{S}}\xspace}
\nc{\bbT}{\ensuremath{\mathbb{T}}\xspace}
\nc{\bbU}{\ensuremath{\mathbb{U}}\xspace}
\nc{\bbV}{\ensuremath{\mathbb{V}}\xspace}
\nc{\bbW}{\ensuremath{\mathbb{W}}\xspace}
\nc{\bbX}{\ensuremath{\mathbb{X}}\xspace}
\nc{\bbY}{\ensuremath{\mathbb{Y}}\xspace}
\nc{\bbZ}{\ensuremath{\mathbb{Z}}\xspace}

\nc{\mrm}[1]{\ensuremath{\mathrm{#1}}\xspace}
\nc{\mit}[1]{\ensuremath{\mathit{#1}}\xspace}
\nc{\mbf}[1]{\ensuremath{\mathbf{#1}}\xspace}
\nc{\mcal}[1]{\ensuremath{\mathcal{#1}}\xspace}
\nc{\msc}[1]{\ensuremath{\mathscr{#1}}\xspace}

\nc{\sub}{\subseteq}
\nc{\too}{\longrightarrow}
\nc{\hook}{\hookrightarrow}
\nc{\hooklongrightarrow}{\lhook\joinrel\longrightarrow}
\nc{\hooklong}{\hooklongrightarrow}
\nc{\hooklongleftarrow}{\longleftarrow\joinrel\rhook}
\nc{\twoheadlongrightarrow}{\relbar\joinrel\twoheadrightarrow}
\nc{\longrightleftarrows}{\ \raisebox{0.3ex}{\(\mathrel{\substack{\xrightarrow{\rule{1em}{0em}} \\[-1ex] \xleftarrow{\rule{1em}{0em}}}}\)}\ }

\renc{\ge}{\geqslant}
\renc{\le}{\leqslant}

\nc{\id}{\mathrm{id}}

\DeclareMathOperator{\Hom}{\on{Hom}}
\nc{\uHom}{\underline{\smash{\Hom}}}

\DeclareMathOperator{\End}{\on{End}}
\DeclareMathOperator{\Sym}{\on{Sym}}
\nc{\uEnd}{\underline{\smash{\End}}}

\nc{\colim}{\varinjlim}
\renc{\lim}{\varprojlim}
\nc{\Cofib}{\on{Cofib}}
\nc{\Fib}{\on{Fib}}
\nc{\initial}{\varnothing}
\nc{\op}{\mathrm{op}}

\DeclareMathOperator*{\fibprod}{\times}

\renc{\setminus}{\smallsetminus}

\usepackage{mathtools}
\newcommand{\thmref}[1]{Theorem~\ref{#1}}

\newcommand{\propref}[1]{Proposition~\ref{#1}}

\renewcommand{\eqref}[1]{(\ref{#1})}

\newcommand{\itemref}[1]{\ref{#1}}

\nc{\pr}{{\mrm{pr}}}

\nc{\A}{\bA}
\nc{\Spec}{\on{Spec}}

\nc{\D}{\on{\mbf{D}}}
\nc{\et}{\mrm{\acute{e}t}}
\renc{\H}{\on{H}}
\renc{\sp}{\mrm{sp}}
\nc{\RGamma}{\on{\bR\Gamma}}
\nc{\R}{\bR}
\renc{\L}{\bL}
\nc{\otimesL}{\otimes^\bL}
\nc{\fibprodR}{\fibprod^\bR}

\nc{\gys}{\mrm{gys}}
\nc{\tr}{\mrm{tr}}

\nc{\Rlim}{\bR\!\!\varprojlim}

\nc{\inftyCat}{\term{$\infty$-category}}
\nc{\inftyCats}{\term{$\infty$-categories}}

\nc{\inftyGrpd}{\term{$\infty$-groupoid}}
\nc{\inftyGrpds}{\term{$\infty$-groupoids}}

\title{Absolute Poincaré duality in étale~cohomology}

\author[A.\,A. Khan]{Adeel A. Khan}

\date{2021-09-19}

\makeatletter
\def\l@subsection{\@tocline{2}{0pt}{4pc}{6pc}{}}
\makeatother

\begin{document}

\begin{abstract}
  We extend Poincaré duality in étale cohomology from smooth schemes to regular ones.
  This is achieved via a formalism of trace maps for local complete intersection morphisms.
\end{abstract}

\maketitle

\setlength{\parindent}{0em}
\parskip 0.75em

Let $S$ be a base scheme and let $\Lambda$ denote the constant sheaf $\bZ/n\bZ$ for an integer $n$ which is invertible on $S$.
For a locally of finite type\footnote{%
  Classically, the $!$-operations were constructed for compactifiable morphisms between quasi-compact quasi-separated schemes (see \cite[Exp.~XVII]{SGA4}), which by Nagata is the class of separated morphisms of finite type.
  In \cite{LiuZheng}, they were extended to locally of finite type morphisms.
} $S$-scheme $X$, define the Borel--Moore homology\footnote{%
  If $S$ is regular and $f : X \to S$ is separated of finite type, then $K_X$ is a dualizing complex on $X$ by \cite[Exp.~XVII, Thm.~0.2]{ILO}.
  Professor Illusie has informed me that this definition of ``Borel--Moore homology'', as cohomology with coefficients in $K_X$, is in fact due to Grothendieck.
  See also \cite[\S 2]{Laumon}.
} of $X$ (relative to $S$) as cohomology with coefficients in $K_X := f^!(\Lambda)$, i.e.,
\[ \H_*(X/S, \Lambda) = \H^{-*}(X, K_X) \]
where $f : X \to S$ is the structural morphism.
Our starting point is the following classical result:

\begin{thm}[Poincaré duality]\label{thm:smooth}
  Let $X$ be a smooth $S$-scheme of relative dimension $d$.
  Then there is a canonical isomorphism
  \[ f^!(\Lambda) \simeq \Lambda(d)[2d] \]
  in the derived category $\D(X_\et, \Lambda)$ of étale sheaves of $\Lambda$-modules on $X$, where $f : X \to S$ is the structural morphism.
  In particular, there is a canonical isomorphism
  \begin{equation}\label{eq:a0s7gu1}
    \H_*(X/S, \Lambda) \simeq \H^{2d-*}(X, \Lambda(d)).
  \end{equation}
\end{thm}

See \cite[Exp.~XVIII, Thm.~3.2.5]{SGA4}, and \cite[Thm.~0.1.4]{LiuZheng} in case $X$ is not separated of finite type.
\thmref{thm:smooth} is proven by constructing a formalism of traces\footnote{%
  Throughout the note, all functors are implicitly derived (wherever necessary).
}
\[ \tr_f : f_!\Lambda(d)[2d] \to \Lambda \]
for flat morphisms $f$ whose geometric fibres are of dimension $\le d$ (see \cite[Exp.~XVIII, Thm.~2.9]{SGA4}).
By adjunction the trace gives rise to a \emph{fundamental class}
\[ [X] \in \H_{2d}(X/S, \Lambda)(-d), \]
and the isomorphism \eqref{eq:a0s7gu1} can be realized as cap product with $[X]$.

In this paper our goal is to prove an ``absolute'' version of \thmref{thm:smooth}, where $X$ is a regular scheme over a regular base scheme $S$, and the morphism $f : X \to S$ is only assumed to be locally of finite type.
To that end we construct a formalism of traces for \emph{local complete intersection} morphisms.

\begin{thmX}\label{thm:gys}
  Let $X$ and $Y$ be schemes on which $n$ is invertible.
  For every local complete intersection morphism $f : X \to Y$ of relative virtual dimension $d$, there is a canonical morphism
  \[ \tr_f : f_!\Lambda(d)[2d] \to \Lambda, \]
  in $\D(Y_\et, \Lambda)$ satisfying the following properties:
  \begin{thmlist}
    \item
    \emph{Functoriality.}
    Given another local complete intersection morphism $g : Y \to Z$ of relative virtual dimension $e$, the composite $g\circ f$ is a local complete intersection morphism of relative virtual dimension $d+e$, and there is a commutative diagram
    \[ \begin{tikzcd}
      g_! f_!\Lambda(d)[2d](e)[2e] \ar{r}\ar[equals]{d}
      & g_!\Lambda (e)[2e] \ar{d}{\tr_g}
      \\
      (g\circ f)_!\Lambda(d+e)[2d+2e] \ar{r}{\tr_{g\circ f}}
      & \Lambda
    \end{tikzcd} \]
    in $\D(Z_\et, \Lambda)$.
    If $f=\id_X$, then $\tr_f = \id : \Lambda \to \Lambda$.

    \item
    \emph{Transverse base change.}
    Given any morphism $q: Y' \to Y$, form the cartesian square
    \begin{equation*}
      \begin{tikzcd}
        X' \ar{r}{g}\ar{d}{p}
          & Y'\ar{d}{q}
        \\
        X \ar{r}{f}
          & Y.
      \end{tikzcd}
    \end{equation*}  
    If this square is Tor-independent (e.g. if $q$ is flat), then $g$ is a local complete intersection morphism of relative virtual dimension $d$, and there is a commutative square
    \[ \begin{tikzcd}
      g_!\Lambda(d)[2d] \ar{r}{\tr_g}\ar[equals]{d}
      & \Lambda \ar[equals]{d}
      \\
      q^* f_!\Lambda(d)[2d] \ar{r}{q^*(\tr_f)}
      & q^*\Lambda
    \end{tikzcd} \]
    in $\D(Y'_\et, \Lambda)$.

    \item\label{item:pure}
    \emph{Purity.}
    Denote by
    \[ \gys_f : \Lambda(d)[2d] \to f^!\Lambda \]
    the morphism in $\D(Y_\et, \Lambda)$ obtained from $\tr_f$ by transposition.
    If $f$ smooth, or if $X$ and $Y$ are regular, then $\gys_f$ is an isomorphism.

    \item
    If $f$ is smooth, then $\tr_f$ agrees with the trace morphism of \cite[Exp.~XVIII, Thm.~2.9]{SGA4} (or rather \cite[Thm.~0.1.4]{LiuZheng} in the non-compactifiable case).

    \item
    If $f$ is a regular closed immersion, then $\gys_f$ coincides with the Gysin morphism $\mrm{Cl}_f : \Lambda \to f^!\Lambda(-d)[-2d]$
    constructed in \cite[\S 2.3]{ILO} and \cite[\S 1]{Azumino} (that is, $\gys_f = \mrm{Cl}_f(d)[2d]$).
    In particular, it refines the local cycle class of \cite[\S 2.2]{Cycle}.
  \end{thmlist}
\end{thmX}

Here local complete intersection (lci) morphisms are defined as in \cite[Exp.~VIII, \S 1, Déf.~1.1]{SGA6}.
For us the relevant description will be as follows: a morphism of schemes is lci if and only if it is locally of finite presentation and has perfect relative cotangent complex of Tor-amplitude $[-1,0]$ (under cohomological grading conventions).
See e.g. \cite[Prop.~2.3.14]{KhanRydh} for this equivalence.\footnote{%
  Surprisingly, this does not appear in \cite{SGA6} or other classical references.
  In fact, it seems to be a common misconception that this requires noetherianness or that it relies on Quillen's conjecture.
}

\begin{rem}
  Fix a base scheme $S$ on which $n$ is invertible.
  From \thmref{thm:gys} we can now read off:
  \begin{defnlist}
    \item
    If $X$ is an lci $S$-scheme of relative virtual dimension $d$ then it admits a \emph{fundamental class}
    \[ [X] \in \H_{2d}(X/S, \Lambda)(-d), \]
    given by the morphism $\gys_p(-d)[-2d] : \Lambda \to p^!\Lambda(-d)[-2d]$, where $p : X \to S$ is the structural morphism.\footnote{%
      If $S$ is a field and $X$ is quasi-projective, then this is the image of the fundamental class in the Chow group $\on{A}_d(X)$ by the cycle class map.
    }

    \item
    If $X$ and $S$ are regular and $d = \dim(X) - \dim(S)$, then $\gys_p : p^!\Lambda \simeq \Lambda(d)[2d]$ gives rise to canonical isomorphisms (``absolute Poincaré duality'')
    \begin{equation*}
      \cap [X] : \H^*(X, \Lambda) \to \H_{2d-*}(X/S, \Lambda)(-d).
    \end{equation*}

    \item
    For any lci morphism $f : X \to Y$ between $S$-schemes, $\gys_f : \Lambda(d)[2d] \to f^!\Lambda$ gives rise to Gysin pull-backs
    \[ f^! : \H_*(Y/S, \Lambda) \to \H_{*+2d}(X/S, \Lambda)(-d). \]

    \item
    For any \emph{proper} lci morphism $f : X \to Y$ of relative virtual dimension $d$, $\tr_f$ gives rise to Gysin push-forwards in cohomology
    \[ f_! : \H^*(X, \Lambda) \to \H^{*-2d}(Y, \Lambda(-d)). \]
  \end{defnlist}
\end{rem}

\begin{rem}
  Claim~\itemref{item:pure} in \thmref{thm:gys} contains in particular the statement that for any closed immersion $i : X \to Y$ between regular schemes $X$ and $Y$, there is an isomorphism $i^!(\Lambda)(d)[2d] \simeq \Lambda$ in $\D(X_\et, \Lambda)$.
  This is Grothendieck's absolute purity conjecture, proven by Gabber (see \cite[Exp.~I, 3.1.4]{SGA5}, \cite{Azumino}, \cite[Exp.~XVI, Thm.~3.1.1]{ILO}).
  However, the proof of \itemref{item:pure} uses this as input, i.e., we do not provide a new proof of absolute purity.
\end{rem}

\section{Deformation to the normal stack}

The main ingredient is \emph{deformation to the normal stack}, a variant of deformation to the normal cone that makes sense not just for closed immersions.

Given an lci morphism $f : X \to Y$ of schemes, the normal stack $N_{X/Y}$ is the ``total space'' of the $(-1)$-shifted cotangent complex $L_{X/Y}[-1]$.
To make sense of this, recall that the total space construction $\sE \mapsto \bV_X(\sE) = \Spec_X(\Sym_{\sO_X}(\sE))$ defines an equivalence between finite locally free sheaves and vector bundles over $X$.
This extends to an equivalence between perfect complexes of Tor-amplitude $[0,1]$ and vector bundle stacks over $X$, so that we can write
\[ N_{X/Y} := \bV_X(L_{X/Y}[-1]). \]
See \cite[\S 1.3]{KhanVirtual}, \cite[\S 2]{BehrendFantechi}, \cite[Exp.~XVIII, \S 1.4]{SGA4}.
In \cite{BehrendFantechi} this is called the ``intrinsic normal cone'' or ``intrinsic normal sheaf'' (they agree for lci morphisms).

If $f$ is a closed immersion, then $L_{X/Y}[-1]$ is just the conormal sheaf in degree zero so $N_{X/Y}$ is just the normal bundle.
In general, $L_{X/Y}[-1]$ will typically have nonzero cohomology in degree $1$, which is why $N_{X/Y}$ will only exist as an algebraic stack.
For example if $f$ is smooth then $L_{X/Y}[-1]$ is the cotangent sheaf in degree $-1$, so $N_{X/Y}$ is the classifying stack $BT_{X/Y}$ of the tangent bundle (viewed as a group scheme over $X$ under addition).
If there is a global factorization of $f$ through a regular immersion $i : X \hook M$ and a smooth morphism $p : M \to Y$, then $N_{X/Y}$ is isomorphic to the stack quotient
\[ N_{X/Y} \simeq [N_{X/M}/i^*T_{M/Y}] \]
where $N_{X/M}$ is the normal bundle of $i$ and $T_{M/Y}$ is the relative tangent bundle of $p$.
No choices are involved in the definitions of $L_{X/Y}$ and $N_{X/Y}$, i.e., they are intrinsic to $f$.

Deformation to the normal stack is an $\A^1$-family of algebraic stacks which deforms $f : X \to Y$ to the zero section $0 : X \to N_{X/Y}$.

\begin{thm}
  Let $f : X \to Y$ be an lci morphism.
  Then there exists a commutative diagram of algebraic stacks
  \begin{equation}\label{eq:D}
    \begin{tikzcd}
      X \ar{r}{0}\ar{d}{0}
      & X \times \A^1 \ar[leftarrow]{r}\ar{d}
      & X \times \bG_m \ar{d}{f\times \id}
      \\
      N_{X/Y} \ar{r}{\hat{i}}\ar{d}
      & D_{X/Y} \ar[leftarrow]{r}{\hat{j}}\ar{d}
      & Y \times \bG_m \ar[equals]{d}
      \\
      Y \ar{r}{0}
      & Y \times \A^1 \ar[leftarrow]{r}
      & Y \times \bG_m
    \end{tikzcd}
  \end{equation}
  where each square is cartesian and Tor-independent.
\end{thm}
\begin{proof}
  See \cite[\S 5.1]{Kresch} and \cite[Thm.~2.31]{Manolache}.
  At the referee's request we include the more ``intrinsic'' construction using derived algebraic geometry mentioned in \cite[\S 1.4]{KhanVirtual} (a more general and detailed version of the following argument will appear in \cite{HekkingKhanRydh}).
  
  Denote by $D_{X/Y} \to Y \times \A^1$ the derived Weil restriction of $f : X \to Y$ along $0 : Y \to Y \times \A^1$.
  Thus $D_{X/Y}$ is a derived stack such that for a derived scheme $T$ over $Y\times\A^1$, the $T$-points of $D_{X/Y}$ are given by
  \begin{equation}\label{eq:oyfg010}
    \Hom_{Y\times\A^1}(T, D_{X/Y})
    \simeq \Hom_{Y}(T \fibprodR_{\A^1} 0, X)
  \end{equation}
  where $T \fibprodR_{\A^1} 0$ is the \emph{derived} fibre over $0$.
  In particular, for every derived scheme $T_0$ over $Y$, we have natural isomorphisms
  \begin{multline*}
    \Hom_{Y}(T_0, D_{X/Y} \fibprodR_{\A^1} 0)
    \simeq \Hom_{Y\times\A^1}(T_0, D_{X/Y})\\
    \simeq \Hom_{Y}(T_0 \fibprodR_{\A^1} 0, X)
    \simeq \Hom_{Y}(T_0, N_{X/Y})
  \end{multline*}
  where $T_0$ is regarded over $Y \times \A^1$ by composing with $0 : Y \to Y \times \A^1$, and the last isomorphism comes from the identification
  \[
    T \fibprodR_{\A^1} 0
    = T \fibprod_{0} 0 \fibprodR_{\A^1} 0
    \simeq \Spec_T(\sO_T \oplus \sO_T[1])
  \]
  with the trivial square-zero extension (in the derived sense) over $T$ and the universal property of the cotangent complex in derived algebraic geometry.
  By the Yoneda lemma it follows that $N_{X/Y}$ is the derived fibre of $D_{X/Y} \to \A^1$ over $0$.
  Similarly, the fibre over $\bG_m$ is $Y \times \bG_m$ since
  \begin{multline*}
    \Hom_{Y\times\bG_m}(T_\eta, D_{X/Y} \fibprod_{\A^1} \bG_m)
    \simeq \Hom_{Y\times\A^1}(T_\eta, D_{X/Y})\\
    \simeq \Hom_{Y}(T_\eta \fibprod_{\bG_m} \bG_m \fibprod_{\A^1} 0, X)
    \simeq \Hom_{Y}(\initial, X)
  \end{multline*}
  is naturally isomorphic to $\Hom_{Y\times\bG_m}(T_\eta, Y\times\bG_m) \simeq \{\ast\}$ for all $T_\eta$ over $Y\times\bG_m$.

  Through \eqref{eq:oyfg010} we get a canonical morphism $X \times \A^1 \to D_{X/Y}$ corresponding to $\id_X \in \Hom_Y(X,X)$, which factors $f\times\id : X \times \A^1 \to Y \times \A^1$.
  The commutativity of the two upper squares in \eqref{eq:D} is witnessed by two isomorphisms in the mapping \inftyGrpds
  \begin{align*}
    \Hom_{Y\times\A^1}(X, D_{X/Y}) &\simeq \Hom_Y(X\times 0\fibprodR_{\A^1} 0, X),\\
    \Hom_{Y\times\A^1}(X\times\bG_m, D_{X/Y}) &\simeq \Hom_Y(X\times \bG_m\fibprod_{\A^1} 0, X) \simeq \{\ast\}.
  \end{align*}
  Both squares are homotopy cartesian since the lower two squares and both vertical composite rectangles are.

  So far we have constructed the diagram \eqref{eq:D} in the \inftyCat of derived stacks.
  To show that $D_{X/Y}$ is algebraic, we can appeal to either of two algebraicity results for derived Weil restrictions.
  The first is \cite[Thm.~5.1.1]{HalpernLeistnerPreygel}, which is stated for mapping stacks but applies in view of the formula for derived Weil restriction in Proposition~5.1.14 of \emph{op. cit}.
  Alternatively, in our lci situation there is a more general and easier result in \cite{HekkingKhanRydh}.
  Briefly, the question of algebraicity is local, so using the local structure of lci morphisms (\cite[Prop.~2.3.14]{KhanRydh}) and the fact that derived Weil restriction commutes with fibred products, it boils down to the case where $X = \bV_Y(\sE)$ is a vector bundle over $Y$, whose derived Weil restriction along $0 : Y \to Y\times\A^1$ is the vector bundle stack
  \[ \bV_{Y\times\A^1}\big(0_*(\sE^\vee)^\vee\big). \]
  This argument also shows that $D_{X/Y}$ is in fact a classical algebraic stack.
  Thus \eqref{eq:D} is a diagram in the ordinary category of algebraic stacks, and homotopy cartesianness of the squares translates to cartesianness and Tor-independence.
\end{proof}

\section{Borel--Moore homology of stacks}

Using the extension of the six operations to algebraic stacks defined in \cite{LiuZheng}\footnote{%
  If we work over a base satisfying some strong hypotheses, which hold e.g. for spectra of finite or separably closed fields, then we can also use the formalism of \cite{LaszloOlsson}; cf. \cite[\S 6.5]{LiuZheng}.
} we can define Borel--Moore homology of an algebraic stack $\sX$ (locally of finite type over some base $S$) again by the formula
\[
  \H_k(\sX/S, \Lambda) = \H^{-k}(\sX, p^!(\Lambda))
\]
where $p : \sX \to S$ is the structural morphism.
Equivalently, these are the homology groups of the complexes
\[\begin{multlined}
  \RGamma(\sX/S, \Lambda)
  := \RGamma(\sX, p^!\Lambda)\\
  \simeq \Rlim_{(T,t)} \RGamma(T, p^* t^!\Lambda)
  \simeq \Rlim_{(T,t)} \RGamma(T, (p\circ t)^!\Lambda)(-d_t)[-2d_t],
\end{multlined}\]
where the homotopy limits are over pairs $(T,t)$ where $T$ is a scheme and $t : T \to \sX$ is a smooth morphism of relative dimension $d_t$.
(By Zariski descent, $T$ can also be taken affine.)

It is straightforward to deduce that the localization exact triangle extends to stacks:

\begin{prop}[Localization]
  If $i : \sZ \to \sX$ is a closed immersion with open complement $j : \sU \to \sX$, then we have an exact triangle
  \[
    \RGamma(\sZ/S, \Lambda)
    \xrightarrow{i_*} \RGamma(\sX/S, \Lambda)
    \xrightarrow{j^!} \RGamma(\sU/S, \Lambda),
  \]
  whence a long exact sequence
  \[
    \cdots
    \to \H_k(\sZ/S, \Lambda)
    \xrightarrow{i_*} \H_k(\sX/S, \Lambda)
    \xrightarrow{j^!} \H_k(\sU/S, \Lambda)
    \xrightarrow{\partial} \H_{k-1}(\sZ/S, \Lambda)
    \to \cdots.
  \]
\end{prop}

For example, consider the closed/open pair $(\hat{i},\hat{j})$ from \eqref{eq:D}.
The boundary map gives rise to a specialization map
\begin{multline}\label{eq:sp}
  \sp_{X/Y} : \RGamma(Y/Y, \Lambda)
  \xrightarrow{\mrm{incl}} \RGamma(Y/Y, \Lambda) \oplus \RGamma(Y/Y, \Lambda)(1)[1]\\
  \simeq \RGamma(Y\times\bG_m/Y, \Lambda)[-1]
  \xrightarrow{\partial} \RGamma(N_{X/Y}/Y, \Lambda).
\end{multline}

We will also need homotopy invariance for vector bundle stacks:

\begin{prop}\label{prop:htp}
  Let $\sE$ be the total space of a perfect complex of Tor-amplitude $[0,1]$ over $X$, say of virtual rank $r$.
  Then there is a canonical isomorphism in the derived category of $\Lambda$-modules
  \[
    \RGamma(\sE/Y, \Lambda)
    \simeq \RGamma(X/Y, \Lambda)(r)[2r].
  \]
  In particular,
  \[
    \H_k(\sE/Y, \Lambda)
    \simeq \H_{k-2r}(X/Y, \Lambda)(r)
  \]
  for all $k\in\bZ$.
\end{prop}
\begin{proof}
  Since the projection $\pi : \sE \to X$ is smooth of relative dimension $r$, we have the Poincaré duality isomorphism
  \[
    \RGamma(\sE, \pi^!f^!\Lambda)
    \simeq \RGamma(\sE, f^!\Lambda)(r)[2r],
  \]
  where there is an implicit $\pi^*$ on the right-hand side.
  This is the homotopy limit over $(T,t)$ of the Poincaré duality isomorphisms
  \[
    \RGamma(T, (\pi \circ t)^!f^!\Lambda)(-d_t)[-2d_t]
    \simeq \RGamma(T, f^!\Lambda)(r)[2r]
  \]
  for the smooth morphism $\pi \circ t : T \to \sX \to Y$ of relative dimension $d_t+r$.\footnote{%
    Note that to form this homotopy limit, we need  the Poincaré duality isomorphism for schemes to be functorial in a homotopy coherent sense; however, this coherence comes for free using a standard t-structure argument, see \cite[Thm.~6.2.9, Rem.~4.1.10]{LiuZheng}.
  }

  Secondly, there is a canonical map
  \[ \pi^* : \RGamma(X, f^!\Lambda) \to \RGamma(\sE, f^!\Lambda) \]
  which (as a consequence of étale descent) can be described as the homotopy limit of the maps $\pi_U^*$, where $\pi_U : \sE \fibprod_Y U \to U$, taken over smooth morphisms $U \to Y$ with $U$ affine.
  Therefore the claim is local on $Y$ and we may assume that the perfect complex defining $\sE$ admits a global resolution, so that $\sE$ is globally the stack quotient $[E^1/E^0]$ of a vector bundle morphism $E^0 \to E^1$.
  In this case $\pi^*$ factors through isomorphisms
  \[
    \RGamma(X, f^!\Lambda)
    \to \RGamma(E^1, f^!\Lambda)
    \gets \RGamma(\sE, f^!\Lambda)
  \]
  by homotopy invariance for the vector bundle $E^1 \to X$ (follows by descent from the case of trivial bundles, see \cite[Exp.~XV, Cor.~2.2]{SGA4}) and for the $E^0$-torsor $E^1 \twoheadrightarrow \sE$ (can be checked after base change to affines, over which vector bundle torsors are split).
\end{proof}

\section{The construction}

We return to the situation of an lci morphism $f : X \to Y$, say of relative virtual dimension $d$.
The $(-1)$-shifted cotangent complex $L_{X/Y}[-1]$ is perfect of Tor-amplitude $[0, 1]$ (of virtual rank $-d$), so \propref{prop:htp} yields a canonical isomorphism
\[
  \RGamma(N_{X/Y}/Y, \Lambda)
  \simeq \RGamma(X/Y, \Lambda)(-d)[-2d].
\]

Combining this with the specialization map \eqref{eq:sp} produces now a canonical map
\begin{equation*}
  \RGamma(Y/Y, \Lambda)
  \xrightarrow{\sp_{X/Y}} \RGamma(N_{X/Y}/Y, \Lambda)
  \simeq \RGamma(X/Y, \Lambda)(-d)[-2d].
\end{equation*}
In particular, the image of the unit $1 \in \RGamma(Y/Y, \Lambda)$ gives rise to a canonical element (a relative fundamental class)
\begin{equation}\label{eq:fund}
  [X/Y] \in \RGamma(X/Y, \Lambda)(-d)[-2d].
\end{equation}
Our Gysin morphism is then the corresponding morphism
\begin{equation}\label{eq:gys}
  \gys_f : \Lambda(d)[2d] \to f^!\Lambda
\end{equation}
in $\D(X_\et, \Lambda)$, and the trace morphism $\tr_f : f_! \Lambda(d)[2d] \to \Lambda$ is its transpose.

It will also be useful to note that these can be refined to natural transformations
\begin{align}
  &\gys_f : f^*(d)[2d] \to f^!\label{eq:natgys}\\
  &\tr_f : f_!f^*(d)[2d] \to \id.\label{eq:nat}
\end{align}
For example, $\tr_f$ is the composite
\[
  f_!f^*(-)(d)[2d]
  \simeq (-) \otimes f_!\Lambda(d)[2d]
  \xrightarrow{\id \otimes \tr_f} (-) \otimes \Lambda
  = \id
\]
where the isomorphism is the projection formula.
Note that when \eqref{eq:gys} is invertible, \eqref{eq:natgys} will also be invertible on dualizable objects in $\D(Y_\et, \Lambda)$ (but not necessarily on arbitrary ones).

\section{Proofs of the asserted properties}

We begin by noting that, in case $f$ is a closed immersion, our construction of the Gysin morphism obviously coincides with that of \cite[\S 3.2]{DegliseJinKhan}, which itself agrees with Gabber's construction \cite[Exp.~XVI, \S 2.3]{ILO} by \cite[4.4.3]{DegliseJinKhan}.
The base change and functoriality properties are proven exactly as in the case of closed immersions, using respectively Tor-independent base change of the deformation space $D_{X/Y}$ (see \cite[Thm.~1.3(ii)]{KhanVirtual}) and the double deformation space associated to lci morphisms $X \to Y \to Z$,
\[
  D_{X/Y/Z} := D_{D_{X/Z} \fibprod_Z Y/D_{X/Z}},
\]
the deformation to the normal stack of the morphism $D_{X/Z} \fibprod_Z Y \to D_{X/Z}$.
See the proof of \cite[Thm.~3.2.21]{DegliseJinKhan}.

Let us show that if $f$ is smooth of relative dimension $d$, then $\gys_f$ is the Poincaré duality isomorphism $f^!(\Lambda) \simeq \Lambda(d)[2d]$.
Form the cartesian square
\[ \begin{tikzcd}
  X \fibprod_Y X \ar{r}{\pr_2}\ar{d}{\pr_1}
  & X \ar{d}{f}
  \\
  X \ar{r}{f}
  & Y
\end{tikzcd} \]
The diagonal morphism $\Delta : X \to X\fibprod_Y X$ is lci of relative virtual dimension $-d$ and the natural transformation $\tr_\Delta : \Delta_!\Delta^*(-d)[-2d] \to \id$ \eqref{eq:nat} gives rise to
\[
  \eta_f :
  \id = \pr_{2,!}\Delta_!\Delta^*\pr_1^*
  \xrightarrow{\tr_\Delta} \pr_{2,!}\pr_1^*(d)[2d]
  \simeq f^*f_!(d)[2d].
\]
We claim that $\eta_f$ and $\tr_f$ form the unit and counit of an adjunction $(f_!, f^*(d)[2d])$.
Indeed, it is easy to check that both composites
\[\begin{tikzcd}[matrix xscale=2, matrix yscale=0.3]
  f_! \ar{r}{f_!(\eta_f)}
  & f_!f^*f_!(d)[2d] \ar{r}{\tr_f \ast f_!}
  & f_!
  \\
  f^*(d)[2d] \ar{r}{\eta_f \ast f^*(d)[2d]}
  & f^*f_!f^*(2d)[4d] \ar{r}{f^* \ast \tr_f(d)[2d]}
  & f^*(d)[2d]
\end{tikzcd}\]
are identity by using the functoriality of the trace for the composite $\pr_1\circ\Delta$ (resp. for the composite $\pr_2\circ\Delta$) and by base change for the trace of $f$.
This argument shows not only that $\gys_f$ is an isomorphism but also that it agrees with the Poincaré duality isomorphisms of \cite[Exp.~XVIII, Thm.~3.2.5]{SGA4} and \cite[Thm.~0.1.4]{LiuZheng}, or equivalently that $\tr_f$ agrees with the trace of \cite[Exp.~XVIII, Thm.~2.9]{SGA4} or \cite[Thm.~0.1.4]{LiuZheng}: indeed, both are counits for the same adjunction.

It remains to show that if $X$ and $Y$ are regular (in which case $f : X \to Y$ is automatically lci), then $\gys_f$ gives the isomorphism $f^!\Lambda \simeq \Lambda(d)[2d]$ asserted in \thmref{thm:gys}\itemref{item:pure}.
But invertibility of $\gys_f$ can be checked after inverse image along a Zariski cover, and by functoriality we have for any open immersion $j : U \hook X$ a commutative diagram
\[ \begin{tikzcd}
  \Lambda(d)[2d] \ar{r}{\gys_j}\ar[equals]{d}
  & j^!\Lambda(d)[2d] \ar{r}{j^!(\gys_f)}
  & j^!f^!\Lambda \ar[equals]{d}
  \\
  \Lambda(d)[2d] \ar{rr}{\gys_{f\circ j}}
  & & (f\circ j)^!\Lambda.
\end{tikzcd} \]
where $\gys_j$ is invertible.
Thus we may localize on $X$ and choose a global factorization through a closed immersion $i : X \hook X'$ and a smooth morphism $p : X' \to Y$.
By functoriality of Gysin morphisms again and the fact that $\gys_p$ is an isomorphism by above, we reduce to the case of a closed immersion between regular schemes (note that $X'$ is still regular).
Finally, since $\gys_f$ agrees with Gabber's construction in this case, the claim now follows from absolute purity \cite[Exp.~XVI, Thm.~3.1.1]{ILO}.

\section{Remarks}

Using the formalism of \cite{LiuZheng}, our construction of the traces $\tr_f$ immediately extends to the case where the schemes $X$ and $Y$ are algebraic stacks.
Absolute Poincaré duality also extends to regular algebraic stacks with the same proof.

We can also allow $X$ and $Y$ to be derived (schemes or stacks), and $f : X \to Y$ to be any \emph{quasi-smooth} morphism.
Indeed, an lci morphism is precisely a quasi-smooth morphism whose source and target happen to be classical (underived).
The construction of \thmref{thm:gys} goes through mutatis mutandis, since the deformation space $D_{X/Y}$ exists in that setting (see \cite[\S 1.4]{KhanVirtual}): it is simply the Weil restriction of $X$ along $0 : Y \hook Y \times \A^1$ (in the derived sense).
For a quasi-smooth morphism, the trace is a kind of categorification of Kontsevich's virtual fundamental class (cf. \eqref{eq:fund}) and gives rise for example to the Gromov--Witten theory of  smooth projective varieties in arbitrary characteristic.
On the other hand, absolute Poincaré duality does not hold for derived schemes whose classical truncations are not regular.

Finally, the construction can be refined from étale cohomology to motivic cohomology.
For this one can use the limit-extended motivic cohomology of algebraic stacks defined in \cite[\S 12]{KhanRavi} as a substitute for \cite{LiuZheng}.
Note that the trace formalism for \emph{flat} maps (as developed in \cite[Exp.~XVIII, Thm.~2.9]{SGA4}) has recently been extended to motivic cohomology by Abe (see \cite{Abe}).
Note that in this setting our proof of absolute Poincaré duality goes through only in equicharacteristic, since absolute purity in motivic cohomology is open in general (see \cite[Thm.~C.1]{DFJKMinus} for the equicharacteristic case).

In the setting of rational and étale motivic cohomology, the results of this paper appeared in a somewhat different form in the preprint \cite{KhanVirtual}.
The present paper is an attempt to give a short and self-contained account without using the language of motives or derived algebraic geometry.

In a future paper, I will explain how to use deformation to the normal stack to generalize Verdier's specialization functor \cite{Verdier}.
This will be combined with a derived version of Laumon's homogeneous Fourier transform \cite{LaumonFourier} to give an analogue of microlocalization in the sense of Kashiwara--Schapira \cite{KashiwaraSchapira} for singular schemes.

\section{Acknowledgments}

I would like to thank Luc Illusie, Jo\"el Riou, and the anonymous referee for helpful comments on previous versions.

These results were first announced during a conference at the Euler International Mathematical Institute in St. Petersburg, Russia in 2019.
I would like to thank the organizers for the invitation to speak there.
At the time, I was partially supported by SFB 1085 Higher Invariants, Universit\"at Regensburg.



\bibliographystyle{halphanum}

Institute of Mathematics,
Academia Sinica,
Taipei 10617,
Taiwan

\end{document}